\documentclass[11pt]{amsart}
\usepackage{amsmath,amssymb,latexsym,dsfont}
\usepackage[left=2.6cm,right=2.6cm,top=2.7cm,bottom=2.7cm]{geometry}

\usepackage{amsmath,amsfonts,amssymb,epsfig, textcomp}
\usepackage{graphicx,tikz}
\usepackage{hyperref, cleveref} 
\usepackage{cases,listings}
\usepackage{color}
\usepackage{mathabx}

\usepackage[pagewise]{lineno}

\newtheorem{theorem}{Theorem}[section]
\newtheorem{lemma}{Lemma}[section]

\newtheorem{proposition}{Proposition}[section]

\newtheorem{remark}{Remark}[section]

\numberwithin{equation}{section}

      \newcommand{\hy}{\hat y}

   \newcommand{\ty}{\widetilde y}

      \newcommand{\cN}{{\mathcal N}}

      \newcommand{\N}{\mathbb{N}}

      \newcommand{\eps}{\varepsilon}
      \newcommand{\mR}{\mathbb{R}}
      \newcommand{\mZ}{\mathbb{Z}}

      \newcommand{\mC}{\mathbb{C}}

      \newcommand{\supp}{\operatorname{supp}}

      \newcommand{\M}{{\mathcal M}}

      \makeatletter
      \def\@setcopyright{}
      \def\serieslogo@{}
      \makeatother

\newcommand{\hmr}[1]{\textcolor{black}{#1}}

\newcommand{\cG}{\mathcal G}

\newcommand{\cP}{\mathcal P}

\newcommand{\cL}{\mathcal L}

\newcommand{\cA}{{\mathcal A}}

\newcommand{\tz}{\widetilde z}

\newcommand{\be}{\begin{equation}}
\newcommand{\ee}{\end{equation}}

\newcommand{\cD}{{\mathcal D}}

\newcommand{\hz}{\hat z}

\title[Cost of fast controls for a KdV system]{Optimal bounds for the cost of fast controls of a KdV system}

\author[H.-M. Nguyen]{Hoai-Minh Nguyen}
\address[H.-M. Nguyen]{Sorbonne Universit\'e, Universit\'e Paris Cit\'e, CNRS, INRIA, \newline \indent 
Laboratoire Jacques-Louis Lions, LJLL, F-75005 Paris, France
}
\email{hoai-minh.nguyen@sorbonne-universite.fr}

\begin{document}

\maketitle 


\begin{abstract} We study the cost of fast controls for a linearized KdV system and a nonlinear KdV system locally, using right Neumann boundary control for non-critical lengths. Since the operator associated with the linearized system is neither self-adjoint nor skew-adjoint, its (known) spectral properties are not directly amenable to the moment method, leaving optimal cost bounds an open problem. We address this difficulty by shifting attention to a related KdV system and deriving the optimal bounds 
from the new one. 
\end{abstract}

\bigskip 



\tableofcontents


\section{Introduction}

This paper investigates the cost of fast control for the linearized and nonlinear Korteweg–de Vries (KdV) equations locally using right Neumann boundary controls. More precisely, concerning the linearized setting, given $L>0$, we consider the following control problem, for $T>0$,  
\begin{equation}\label{sys-KdV}\left\{
\begin{array}{cl}
y_t  + y_x  + y_{xxx}    = 0 &  \mbox{ in } (0, T) \times (0, L), \\[6pt]
y(\cdot, 0) = y(\cdot, L) = 0 & \mbox{ in }   (0, T), \\[6pt]
y_x(\cdot, L) = u & \mbox{ in }  (0, T), \\[6pt]
y(0, \cdot)  = y_0  &  \mbox{ in }  (0, L). 
\end{array}\right.
\end{equation}
Here  $y$ is the state, $y_0 \in L^2(0, L)$ is an initial datum,  and $u \in L^2(0, T)$ is a control.  The KdV equation has been introduced by Boussinesq \cite{1877-Boussinesq} and Korteweg
and de Vries \cite{KdV} as a model for the propagation of surface water waves along a
channel. This equation also furnishes a very useful nonlinear approximation model
including a balance between weak nonlinearity and weak dispersive effects, see, e.g.,~\cite{Whitham74, Miura76, Kato83}.  


For initial and final data in $L^2(0, L)$, and controls in $L^2(0, T)$, Rosier~\cite{Rosier97} proved  that the linearized KdV system \eqref{sys-KdV} with the right Neumann controls
is exactly controllable for any positive time $T$ provided that the length $L$ is not critical,
i.e., $L \notin \cN$, where \begin{equation}\label{def-cN}
\cN : = \left\{ 2 \pi \sqrt{\frac{k^2 + kl + l^2}{3}}; \, k, l \in \N_*\right\}.
\end{equation}
Rosier also established that when $L \in \cN$,  the linearized system is not null-controllable for any time $T>0$. More precisely, he showed
that there exists a non trivial, finite dimensional subspace $\M$  of $L^2(0, L)$  such that its orthogonal space is reachable from $0$ for small time whereas  $\M \hmr{\setminus \{0\}}$ is not for any time.

Although the control system \eqref{sys-KdV} and its nonlinear version have been studied in a quite few works for some time, see, e.g., the surveys \cite{RZ09, Cerpa14}, to the best of our knowledge, there are no results on the cost of fast controls for this system when $L$ is not critical. Part of the reason will be explained below. The goal of this paper is to give an answer to this problem. 

Set, for $T>0$, 
\be
X_T = C([0, T]; L^2(0, L)) \cap L^2((0, T); H^1_0(0, L)), 
\ee
and equip this space with the corresponding norm, i.e., 
$$
\| y \|_{X_T} =\| y \|_{C([0, T]; L^2(0, L))} + \| y \|_{L^2((0, T); H^1(0, L))}. 
$$ 
Then $X_T$ is a Banach space.  It is known that  given $y_0 \in L^2(0, L)$ and $u \in L^2(0, T)$, there exists a unique weak solution $y \in X_T$ of \eqref{sys-KdV} (see, e.g., \cite{Rosier97,CC04}). 

\medskip 
We next state the main results of this paper. Concerning the linearized KdV system \eqref{sys-KdV}, we obtain the following result, which implies the optimal bound for the cost of fast controls.  

\begin{theorem} \label{thm-LN} Assume that $L \not \in \cN$ and $0< T < T_0$. We have
\begin{itemize}
\item[a)] Given $y_0 \in L^2(0, L)$,  there exists $ u \in L^2(0, T)$ such that  $y(T, \cdot) = 0$, where $y \in X_T$ is the unique solution of \eqref{sys-KdV}. Moreover, 
\be
\| u\|_{L^2(0, T)} \le C_1 e^{\frac{c_1}{T^{1/2}}} \| y_0 \|_{L^2(0, L)}.  
\ee

\item[b)] There exist $y_0 \in H^1_0(0, \hmr{L})$ such that for $u \in L^2(0, T)$ being such that $y(T, \cdot) = 0$, where $y \in X_T$ is the unique solution of \eqref{sys-KdV}, one has, if $T_0$ is sufficiently small \footnote{The smallness of $T_0$ depending only on $L$.},
\be
\| u\|_{L^2(0, T)} \ge C_2 e^{\frac{c_2}{T^{1/2}}} \| y_0 \|_{L^2(0, L)}. 
\ee
\end{itemize}
Here $c_1, C_1$, $c_2, C_2$ are positive constants depending only on $L$ and $T_0$.  
\end{theorem}

Concerning the nonlinear KdV system, we \hmr{establish} the following result. 

\begin{theorem} \label{thm-NL} Assume that $L \not \in \cN$ and $0< T < T_0$.  The following two facts hold:
\begin{itemize}
\item[a)] There exists a positive constant $\eps_0 = \eps_0 (T)$  such that given $y_0 \in L^2(0, L)$ with $\| y_0\|_{L^2(0, L)}  \le \eps_0$,  there exists $ u \in L^2(0, T)$ such that $y(T, \cdot) = 0$, where $y \in X_T$ is the unique solution of the nonlinear KdV system 
\be \label{sys-KdV-NL}
\left\{ \begin{array}{cl}
y_t + y_x + y_{xxx} + y y_x= 0 \mbox{ in } (0, T) \times (0, L), \\[6pt]
y(\cdot, 0) = y(\cdot, L) = 0 \mbox{ in } (0, T), \\[6pt]
y_x(\cdot, L) = u(t) \mbox{ in } (0, T), \\[6pt]
y(0,  \cdot) \hmr{= y_0} \mbox{ in } (0, L).
\end{array} \right. 
\ee
Moreover, 
\be
\| u\|_{L^2(0, T)} \le C_1 e^{\frac{c_1}{T^{1/2}}} \| y_0 \|_{L^2(0, L)}. 
\ee
\item[b)] There exists $\varphi \in H_0^1(0, L)$  and $\eps_1 = \eps_1 (T) > 0$ such that for $u \in L^2(0, T)$ with $\| u \|_{L^2(0, T)} \le \eps_1$ being such that $y(T, \cdot) = 0$ where $y \in X_T$ is the unique solution of \eqref{sys-KdV} with $y_0 = \eps \varphi$ and $0 < \eps < \eps_1$, one has, if $T_0$ is sufficiently small \footnote{The smallness of $T_0$ depending only on $L$.},
\be
\| u\|_{L^2(0, T)} \ge C_2 e^{\frac{c_2}{T^{1/2}}} \| y_0 \|_{L^2(0, L)}. 
\ee
\end{itemize}
Here $c_1, C_1$, $c_2, C_2$ are positive constants depending only on $L$ and $T_0$.  
\end{theorem}

We next briefly discuss the analysis of these results. 
One standard way to obtain estimates on the cost of fast control of \eqref{sys-KdV} is to use the moment method
introduced by Fattorini and Russell \cite{FR71}. The idea applied to \eqref{sys-KdV} can be described as follows. Given $\Phi_T \in L^2(0, L)$, let $\Phi \in X_T$ be the unique solution of the backward linearized KdV system 
 \be \label{sys-KdV-Phi}
\left\{ \begin{array}{cl}
\Phi_t + \Phi_x + \Phi_{xxx} = 0 \mbox{ in } (0, T) \times (0, L), \\[6pt]
\Phi(\cdot, 0) = \Phi(\cdot, L) = \Phi_x (\cdot, 0) = 0 \mbox{ in } (0, T),   \\[6pt]
\Phi(T, \cdot) = \Phi_T \hmr{\mbox{ in } (0, L)}. 
\end{array} \right. 
\ee
Multiplying the equation of $y$ by $\Phi$ and integrating by parts yield 
\be
\int_0^L \Big( \Phi(T, x) y(T, x) - \Phi(0, x) y (0, x) \Big) \, dx = \int_0^T \Phi_x(t, L) u (t) \, dt. 
\ee
One then derives that $u \in L^2(0, T)$ is a control which steers $y_0$ at time $0$ to $0$ at time $T$ if and only if 
\be \label{KdV-identity}
\int_0^T \Phi_x(t, L) u (t) \, dt =  - \int_0^L \Phi(0, x) y_0 (x) \, dx, 
\ee
for all $\Phi_T$ belonging to a \hmr{dense} subset of $L^2(0, L)$. This is the starting point of the moment method.

Set 
\be
\cD(\cA) = \Big\{ \varphi \in H^3(0, L); \varphi (0) = \varphi(L) = \varphi'(0) = 0  \Big\}, 
\ee
and 
\be
\cA \varphi = - \varphi''' - \varphi \mbox{ for } \varphi \in \cD(\cA). 
\ee
Let $\lambda$ be an eigenvalue of $\cA$ and let $\varphi_\lambda$ be a corresponding eigenfunction, i.e., $\varphi_\lambda \in \cD(\cA)$ and $- \varphi_\lambda''' - \varphi_\lambda = \lambda \varphi_\lambda$. Set 
$$
\Phi(t, x) = e^{\lambda t} \varphi_\lambda (x) \mbox{ in } (0, T) \times (0, L). 
$$
One can check that $\Phi$ is a solution of \eqref{sys-KdV-Phi} with $\Phi_T = e^{\lambda T} \varphi_\lambda (x)$. 
From \eqref{KdV-identity}, one has 
\be \label{KdV-identity}
\varphi_{\lambda}' (L) \int_0^T   e^{\lambda t}  u (t) \, dt =  - \int_0^L \varphi_\lambda(x) y_0 \, dx.  
\ee
If one knew that the set of eigenfunctions of $\cA$ form a dense set of $L^2(0, L)$, this happens if $\cA$ would be self-adjoint or skew adjoint with compact resolvent, then $u$ would be a control if and only if \eqref{KdV-identity} held for all eigenfunctions of $\cA$. One could then construct a function $u$ such that this property holds to obtain an upper bound on the cost, and use the complex analysis dealing with the zeros of holomorphic functions as suggested by Coron and Guerrero \cite{CG05} to obtain a lower bound. 

Unfortunately, the spectral information  $\big(\cA, \cD(\cA) \big)$ is not in a ready state to apply the moment method. Since $\cA$ is neither self-adjoint nor skew-adjoint, little is known about the full range of its eigenvalues and the corresponding eigenfunctions. It is known that the spectrum of $\cA$
is discrete, see, e.g.,  \cite[Lemma 2.1]{CKN-24}. Nevertheless, it is not known whether 
the set of eigenfunctions form a dense subset of $L^2(0, L)$. For these reasons, the cost of fast control for \eqref{sys-KdV} has not been derived and remains an open problem up to date.  

The goal of this paper is to answer this problem. Our idea, which is simple but can be easily applied to other settings, is as follows.
We connect the control system \eqref{sys-KdV} with another related control KdV system, for which the moment method can be applied, and then derive the cost information from the latter. More precisely, we consider the following control system 
\be \label{sys-KdV-P}
\left\{ \begin{array}{cl}
y_t + y_x + y_{xxx} = 0 \mbox{ in } (0, T) \times (0, L), \\[6pt]
y(\cdot, 0) = y(\cdot, L) = 0 \mbox{ in } (0, T), \\[6pt]
y_x(\cdot, L) - y_x(\cdot, 0) = v(t) \mbox{ in } (0, T), \\[6pt]
y(0,  \cdot) = y_0  \mbox{ in } (0, L).
\end{array} \right. 
\ee
The boundary control $y_x(t, L) = u(t)$ is considered in \eqref{sys-KdV} \hmr{whereas} the boundary control $y_x(\cdot, L) - y_x(\cdot, 0) = v(t)$ is used in \eqref{sys-KdV-P}.  Concerning \eqref{sys-KdV-P}, set \footnote{The index $m$ stands for the modification.}
\be
\cD(\cA_m) = \Big\{w \in H^3(0, L); w(0) = w(L) = 0, \; w'(0) = w'(L) \Big\}, 
\ee
and 
\be
\cA_m w = - w''' - w' \mbox{ for } w \in \cD(\cA_m).  
\ee
Then $\cA_m$ is a skew-adjoint operator with compact resolvent. The moment method can be applied to \eqref{sys-KdV-P}. The idea is then to use the information on the cost of \eqref{sys-KdV-P} and estimates on solutions of \eqref{sys-KdV-P} (and also \eqref{sys-KdV}) to derive the cost for the control of \eqref{sys-KdV} by noting that  
$$
u(t) = v(t) + y_x (t, 0) \mbox{ in } (0, T)
$$
is a control of \eqref{sys-KdV} which steers the initial state $y_0$ at time $0$ to $0$ at time $T$ if and only if $v$ is a control of \eqref{sys-KdV} which steers the initial state $y_0$ at time $0$ to $0$ at time $T$. 
\hmr{To derive optimal estimate for the cost of fast controls for the original system, sharp estimates for the KdV systems play an important role. }

It is noted in \cite{Ng-KdV-S,Ng-S-Schrodinger} inspired from \cite{Ng-Riccati,CoronNg17} that if the cost of fast controls is well-understood then one can do stabilization in finite time for several control systems modeled by a strongly continuous group. 
Materials on control systems in the setting of 
strongly continuous semi-groups can be found in \cite{CZ95,BDDM07,Coron07,TW09}. Another way to reach the stabilization in finite time via backstepping approach is to derive suitable estimates for the kernels involved in the method as suggested by Coron and Nguyen \cite{CoronNg17}. This has been also applied to KdV systems, see \cite{Xiang18,Xiang19}.

The paper is organized as follows. In \Cref{sect-Pre}, we recall several known results which are used in the proof of \Cref{thm-LN,thm-NL}. In particular, the cost bounds of fast controls for the system \eqref{sys-KdV-P} is given.  A proof of this result in the spirit of \cite{Ng-cost-fractional-SH} is given in the appendix. The proofs of \Cref{thm-LN,thm-NL} are then given in \Cref{sect-thm-LN,sect-thm-NL}, respectively. 

\section{Preliminaries} \label{sect-Pre}

In this section, we collect some known results which are used in the proofs of \Cref{thm-LN,thm-NL}. Concerning the spectrum of $\cA_m$, it is known that, see, e.g., \cite{CC09-DCDS}, the set of eigenvalues of $\cA_m$ can be indexed by $(i \lambda_k)_{k \in \mZ \setminus \{0\}}$. Denote $(\varphi_k)_{k \in \mZ \setminus \{0\}}$ the corresponding eigenfunctions which form a basis of $L^2(0, L)$. The following information on $(\lambda_k)_{k \in \mZ \setminus \{0\}}$ and ($\varphi_k)_{k \in \mZ \setminus \{0\}}$ are given in \cite[Proposition 1 and page 660]{CC09-DCDS} (see also \cite[Section 2]{CL14}). 


\begin{lemma} \label{lem1} We have 
\be
\lambda_k = \frac{8 \pi^3 k^3}{L^3} + O(k^2) \mbox{ as } k \to + \infty, 
\ee
and, for large positive $k$, 
\begin{multline}
\varphi_k(x) 
= \alpha_k \Big[e^{- i a_k x } \Big( \cosh \big( (3 a_k^2 - 1)^{1/2} x\big)\Big)  \\[6pt]  + \frac{e^{3 i a_k L} - \cosh \big( (3 a_k^2 - 1)^{1/2} L\big) }{\sinh \big( (3 a_k^2 - 1)^{1/2} L\big)} \sinh \big( (3 a_k^2 - 1)^{1/2} x\big) \Big) - e^{2 i a_k x} \Big] \mbox{ in } [0, L],
\end{multline}
with $a_k = \frac{5 \pi }{6 L} + \frac{k \pi }{L}  + O (1/k)$ and $\lim_{k \to + \infty} \alpha_k = \frac{1}{\sqrt{L}}$. Moreover, for $k \ge 1$,  
\be
\lambda_{-k} = - \lambda_k  \quad \mbox{ and } \quad  
\varphi_{-k} = \bar \varphi_k \mbox{ in } [0, L].  
\ee
 \end{lemma}

It is useful to note that in the case $L \not \in \cN$, one has 
\be \label{cond-N-1}
\varphi_k'(L) \neq 0 \mbox{ for all } k \in \mZ \setminus \{0 \}.
\ee
As a consequence of \Cref{lem1}, \hmr{there exists a constant $\gamma \ge 1$ depending only on $L$, such that for $L \not \in \cN$},  
\be \label{cond-N-2}
\hmr{\gamma^{-1} |k|} \le |\varphi_k'(L)|  \mbox{ for } k \in \mZ \setminus \{0 \}. 
\ee
\hmr{and, for every $L > 0$, 
\be \label{cond-N-3}
|\varphi_k'(L)| \le \gamma |k| \mbox{ for } k \in \mZ \setminus \{0 \}. 
\ee}

Using this and the moment \hmr{approach}, one can prove the following result. 
\begin{proposition}\label{pro1} Assume that $L \not \in \cN$ and $0< T < T_0$. We have
\begin{itemize}
\item[a)] Given $y_0 \in L^2(0, L)$,  there exists $v \in L^2(0, T)$ such that  $y(T, \cdot) = 0$, where $y \in \hmr{C([0, T]; L^2(0, L))}$ is the unique solution of \eqref{sys-KdV-P}. Moreover, 
\be \label{pro1-cl1}
\| v\|_{L^2(0, T)} \le C_1 e^{\frac{c_1}{T^{1/2}}} \| y_0 \|_{L^2(0, T)}.    
\ee
\hmr{Assume in addition that $y_0 \in H_0^1(0, L) \cap H^2(0, L)$. Then the following fact also holds  
\be \label{pro1-cl2}
\|y_x (\cdot, 0) \|_{L^2(0, T)} \le C_1 e^{\frac{c_1}{T^{1/2}}} \| y_0 \|_{H^2(0, T)}. 
\ee
}

\item[b)] Let  $y_0 = \varphi_1$ and let $v \in L^2(0, T)$ be such that $y(T, \cdot) = 0$, where $y \in \hmr{  C([0, T]; L^2(0, L))}$  is the unique solution of \eqref{sys-KdV-P}.  One has 
\be \label{pro1-cl3}
\| v \|_{L^2(0, T)} \ge C_2 e^{\frac{c_2}{T^{1/2}}} \| y_0 \|_{L^2(0, T)}. 
\ee
\end{itemize}
Here $c_1, C_1$, $c_2, C_2$ are positive constants depending only on $L$ and $T_0$.  
\end{proposition}

Assertion \hmr{\eqref{pro1-cl1}} and a slightly different version of Assertion \hmr{\eqref{pro1-cl3}} 
were proved by Lissy \cite[Theorem 3.2]{Lissy14}.  In the appendix, we \hmr{provide} a proof of the results \hmr{stated} here in the spirit of our previous work \cite{Ng-cost-fractional-SH} (see also \cite{CG05,Lissy15,Lissy17}). \hmr{Assertion \eqref{pro1-cl2} is new to our knowledge and its proof is also given there. }

\medskip 
Concerning the estimates for solutions of \eqref{sys-KdV-P}, \hmr{we have the following result.}

\begin{lemma} \label{lem-KdV1} Let $L > 0$, $0 < T < T_0$, and let $y_0 \in \hmr{H_0^1 (0, L)}$. There exists a unique weak solution $y \in \hmr{  C([0, T]; L^2(0, L))}$ of the system 
\be \label{lem-KdV1-sys}
\left\{\begin{array}{c}
y_t + y_x + y_{xxx}  = 0 \quad \mbox{ in } (0, T) \times (0, L), \\[6pt]
y(\cdot, 0) = y (\cdot, L) = 0  \mbox{ in } (0, T), \\[6pt]
 y_x(\cdot, L) - y_x(\cdot, 0) = \hmr{0} \mbox{ in } (0, T), \\[6pt]
y(0, \cdot) = y_0 (\cdot)  \mbox{ in } (0, L).   
\end{array} 
\right. 
\ee
Moreover, 
\be \label{lem-KdV1-est}
\|y_x(\cdot, L) \|_{L^2(0, T)} \le C \| y_0\|_{H^1(0, L)}, 
\ee
for some positive constant $C$ independent of $f$, $h$, $y_0$, and $T$. 
\end{lemma}

{\color{black}
\begin{proof} We have 
$$
y(t, x) = \sum_{k \in \mZ \setminus \{0 \}} a_k e^{i \lambda_k t} \varphi_k (x) \mbox{ in } (0, T) \times (0, L),  
$$
where 
$$
a_k = \int_0^L y_0 \bar \varphi_k \, dx. 
$$
It follows that 
$$
y_x(t, L) = \sum_{k \in \mZ \setminus \{0 \}}  a_k \varphi_k'  (L) e^{i \lambda_k t}. 
$$
By Ingham's inequality, 
\be \label{lem-KdV1-p1}
\| y_x(\cdot, L) \|_{L^2(0, T)}^2 \le C \sum_{k \in \mZ \setminus \{0 \}}  |a_k \varphi_k'  (L)|^2 \mathop{\le}^{\eqref{cond-N-3}} C \sum_{k \in \mZ \setminus \{0 \}} |k a_k|^2. 
\ee
Here and in what follows, $C$ denotes a positive constant depending only on $T_0$ and $L$ and can change from one place to another. 

We claim that  
\be \label{lem-KdV1-p2}
\sum_{k \in \mZ \setminus \{0 \}} |k a_k|^2  \le C \| y_0'\|_{L^2(0, L)}^2. 
\ee
Assertion \eqref{lem-KdV1-est} then follows from \eqref{lem-KdV1-p1} and \eqref{lem-KdV1-p2}. 

It remains to prove \eqref{lem-KdV1-p2}. We have, since $\varphi_k''' + \varphi_k' = i \lambda_k \varphi_k$ in $(0, L)$, 
\be  \label{lem-KdV1-p3}
\int_0^L y_0 \bar \varphi_k \, dx  = - \frac{1}{i \lambda_k } \int_0^L y_0 (\bar \varphi_k''' + \bar \varphi_k) \, dx , 
\ee
and, by integration by parts, 
\be  \label{lem-KdV1-p4}
\int_0^L y_0 \bar \varphi_k'''  \, dx = - \int_0^L y_0' \bar \varphi_k'' \, dx , 
\quad \mbox{ and } \quad 
\int_0^L y_0 \bar \varphi_k' \, d x  = - \int_0^L y_0' \bar \varphi_k \, dx. 
\ee

By \Cref{lem1}, we have
\begin{multline*}
\varphi_k''(x) =  4 \alpha_k a_k^2 e^{- 2 i a_k x} + \alpha_k \left(\sqrt{3 a_k^2 - 1} - i a_k x \right)^2  e^{3 i a_k L }   e^{- i a_k x}  e^{\sqrt{3 a_k^2 - 1} (x-L)} \\[6pt]
+ \alpha_k \left(\sqrt{3 a_k^2 - 1} + i a_k x \right)^2 e^{- i a_k x} e^{-\sqrt{3 a_k^2 - 1} x}+  O(k) \mbox{ in } L^2(0, 1)\\[6pt]
= 4 \alpha_k a_k^2 e^{- 2 i a_k x} + \alpha_k \left(\sqrt{3 a_k^2 - 1} - i a_k x \right)^2  e^{3 i a_k L }   e^{- i a_k x}   +  \alpha_k \left(\sqrt{3 a_k^2 - 1} + i a_k x \right)^2   e^{- i a_k x}  \\[6pt]
+ \alpha_k \left(\sqrt{3 a_k^2 - 1} - i a_k x \right)^2  e^{3 i a_k L }  e^{- i a_k x}  \Big( e^{\sqrt{3 a_k^2 - 1} (x-L)} -1 \Big) \\[6pt]
+ \alpha_k \left(\sqrt{3 a_k^2 - 1} + i a_k x \right)^2 e^{- i a_k x} \Big(e^{-\sqrt{3 a_k^2 - 1} x} - 1) +  O(k) \mbox{ in } L^2(0, 1)\\[6pt]
:= f_{1, k} + f_{2, k} + f_{3, k} + f_{4, k} +  f_{5, k} + f_{6, k} \mbox{ in } L^2(0, 1). 
\end{multline*}
Since 
$$
\left| \frac{1}{k^2} \int_0^L y_0' \bar \varphi_k'' \right|^2 = \left| \frac{1}{k^2} \int_0^L y_0' \sum_{l=1}^6 \bar f_{l, k} \right|^2 \le C \sum_{l=1}^6 \left| \frac{1}{k^2} \int_0^L y_0' \bar f_{l, k} \right|^2, 
$$
we derive from \Cref{lem1} that 
$$
\sum_{k \in \mZ \setminus \{0 \}} \left| \frac{1}{k^2} \int_0^L y_0' \bar \varphi_k'' \right|^2 \le C \| y_0' \|^2_{L^2(0, L)}. 
$$
Since 
$$
\sum_{k \in \mZ \setminus \{0 \}}  \left| \int_0^L y_0' \bar \varphi_k \right|^2 \le C \| y_0' \|^2_{L^2(0, L)},  
$$
the claim follows from \eqref{lem-KdV1-p3} and \eqref{lem-KdV1-p4}. 

The proof is complete. 
\end{proof}
}

Concerning the estimates for solutions of \eqref{sys-KdV}, we recall the following result, see, e.g.,  \cite{CKN-24} and \cite{Ng-KdV25}. 

\begin{lemma} \label{lem-KdV2} Let $L > 0$, $0 < T < T_0$, and let $y_0 \in L^2(0, L)$, $f \in L^1((0, T); L^2(0, L))$, and $h \in L^2(0, T)$. There exists a unique weak solution $\hmr{y} \in X_T$ of the system 
\be \label{lem-KdV1-sys}
\left\{\begin{array}{c}
y_t + y_x + y_{xxx}  = f \quad \mbox{ in } (0, T) \times (0, L), \\[6pt]
y(\cdot, 0) = y (\cdot, L) = 0  \mbox{ in } (0, T)\\[6pt] 
y_x(\cdot, L)  = h \mbox{ in } (0, T), \\[6pt]
y(0, \cdot) = y_0 (\cdot)  \mbox{ in } (0, L).   
\end{array} 
\right. 
\ee
Moreover, 
\be \label{lem-KdV1-est}
\| y\|_{X_T} + \|y_x(\cdot, 0) \|_{L^2(0, T)} \le C \Big(\| y_0\|_{L^2(0, L)} +  \| f\|_{L^1((0, T)); L^2(0, T))} + \| h \|_{L^2(0, T)} \Big), 
\ee
for some positive constant $C$ independent of $f$, $h$, $y_0$, and $T$. 
\end{lemma}

\section{The cost of fast controls of the linearized KdV system - Proof of \Cref{thm-LN}} \label{sect-thm-LN}

 The proof is divided into two steps.

\medskip 
\noindent {\it Step 1:} Proof of Assertion $a)$. \hmr{We first take $u= 0$ in the time interval $(0, T/2)$. By the regularizing effect the KdV equation, it holds that $y(T/2, \cdot) \in H_0^1(0, L) \cap H^2(0, L)$ and 
\be \label{thm-LN-Step1-p0}
\| y(T/2, \cdot) \|_{H^2(0, L)} \le \frac{C}{T^2} \|y_0 \|_{L^2(0, L)},  
\ee
where $y \in X_{T/2}$ is the unique solution of \eqref{sys-KdV}. 
}
By Assertion a) of \Cref{pro1}, there exists a control $v \in L^2(T/2, T)$ of \hmr{
\be \label{thm-LN-sys-KdV-P}
\left\{ \begin{array}{cl}
\ty_t + \ty_x + \ty_{xxx} = 0 \mbox{ in } (T/2, T) \times (0, L), \\[6pt]
\ty(\cdot, 0) = \ty(\cdot, L) = 0 \mbox{ in } (T/2, T), \\[6pt]
\ty_x(\cdot, L) - \ty_x(\cdot, 0) = v(t) \mbox{ in } (T/2, T), \\[6pt]
\ty(0,  \cdot) = y (T/2, \cdot)  \mbox{ in } (0, L).
\end{array} \right. 
\ee
}
such that 
\be \label{thm-LN-Step1-p1}
\ty(T, \cdot) = 0 \mbox{ in } (0, L), 
\ee
\be \label{thm-LN-Step1-p2}
\| v\|_{L^2(T/2, T)} \mathop{\le}^{\eqref{pro1-cl1}} C e^{\frac{c}{T^{1/2}}} \| y(T/2, \cdot) \|_{L^2(0, T)}, 
\ee
\hmr{and 
\be \label{thm-LN-Step1-p3}
\| \ty_x(\cdot, L)\|_{L^2(T/2, T)} \mathop{\le}^{\eqref{pro1-cl2}} C e^{\frac{c}{T^{1/2}}} \| y(T/2, \cdot) \|_{H^2(0, T)}. 
\ee}

Here and \hmr{in} what follows in this proof, $c$ and $C$ denote positive constants depending only on $L$ and can change from one place to another.  
Set 
$$
u(t) = \left\{\begin{array}{cl} 0 &  \mbox{ in } (0, T/2), \\[6pt]
v(t) + \ty_x(t, 0) &  \mbox{ in } (T/2, T), 
\end{array} \right. 
$$
\hmr{and let $y_1 \in X_T$ be the unique solution of the system} 
\be  \label{thm-LN-Step1-p4}
\left\{ \begin{array}{cl}
y_{1, t} + y_{1, x} + y_{1, xxx} = 0 \mbox{ in } (0, T) \times (0, L), \\[6pt]
y_1(\cdot, 0) = y_1(\cdot, L) = 0 \mbox{ in } (0, T), \\[6pt]
y_{1, x}(\cdot, L)  = u (t) \mbox{ in } (0, T), \\[6pt]
y_1(0,  \cdot) = y_0 \mbox{ in } (0, L).
\end{array} \right. 
\ee
From \hmr{\eqref{thm-LN-Step1-p0}}, \eqref{thm-LN-Step1-p2}, and \eqref{thm-LN-Step1-p3}, we obtain  
\be\label{thm-LN-Step1-p5}
\| u \|_{L^2(0, T)} \le C e^{\frac{c}{T^{1/2}}} \| y_0 \|_{L^2(0, T)}. 
\ee
Assertion $a)$ now follows from \eqref{thm-LN-Step1-p1}, \eqref{thm-LN-Step1-p4}, and \eqref{thm-LN-Step1-p5} \hmr{after noting that $y_1(t, \cdot) = y(t, \cdot)$ in $(0, T/2)$ and $y_1(t, \cdot) = \ty (t, \cdot)$ for $t \in (T/2, T)$}. 

\medskip 
\noindent {\it Step 2}: Proof of Assertion $b)$ with $\varphi = \varphi_1$. Recall that $\varphi_1$ is the eigenfunction of $\cA_m$. Let $u \in L^2(0, T)$ be such that 
\be
y(T, \cdot) = 0 \mbox{ in } (0, L), 
\ee
where $y \in X_T$ is the unique solution of \eqref{sys-KdV}.  Set 
\be \label{pro1-Step2-p1}
v(t) = u(t) - y_x(t, 0) \mbox{ in } (0, T). 
\ee
Then 
$$
\left\{ \begin{array}{cl}
y_t + y_x + y_{xxx} = 0 \mbox{ in } (0, T) \times (0, L), \\[6pt]
y(\cdot, 0) = y(\cdot, L) = 0 \mbox{ in } (0, T), \\[6pt]
y_x(\cdot, L) - y_x(\cdot, 0)  = v (t) \mbox{ in } (0, T), \\[6pt]
y(0,  \cdot) = y_0 \mbox{ in } (0, L).
\end{array} \right. 
$$
By Assertion b) of \Cref{pro1}, we have 
$$
\| v\|_{L^2(0, T)} \ge C e^{\frac{c}{T^{1/2}}} \| y_0\|_{L^2(0, T)}, 
$$
which yields, by \eqref{pro1-Step2-p1},   
\be \label{pro1-Step2-p2}   
\| u\|_{L^2(0, T)} + \|y_x(\cdot, 0) \|_{L^2(0, T)} \ge C e^{\frac{c}{T^{1/2}}} \| y_0\|_{L^2(0, T)}. 
\ee
Since,  by \Cref{lem-KdV2},    
\be
\| y_x(t, 0)\|_{L^2(0, T)} \le C  \Big( \| y_0 \|_{L^2(0, T)} + \| u \|_{L^2(0, T)}  \Big), 
\ee
we derive from \eqref{pro1-Step2-p2} that for $T_0$ sufficiently small 
\be 
\| u \|_{L^2(0, T)} \ge C e^{\frac{c}{T^{1/2}}} \| y_0\|_{L^2(0, T)},  
\ee
which yields  Assertion $b)$. 
\qed 

\begin{remark} \rm The constants $c_1, C_1$, $c_2$, $C_2$ can be obtained explicitely as a function of $L$ from the proof. 
\end{remark}

\begin{remark} \rm Fix $0 < \ell_1 < \ell_2 < + \infty$ and let $\ell_1 \le L \le \ell_2$. One can choose the constant $c_1$ independent of $L$ but depends on $\ell_1$ and $\ell_2$. Nevertheless, the constant $C_1 = C_1 (L)$ blows up as $\mbox{dist} (L, \cN) \to 0$. This blow-up can be quantified. Similar facts hold for $c_2$ and $C_2$. See \Cref{rem-pro1}. 
\end{remark}

\begin{remark} \rm In \cite{KX21}, the authors showed that there exists a positive constant computable as a function of $L$ when $L \not \in \cN$ for the observability inequality of \eqref{sys-KdV} when $T \ge \tau_0$ for some sufficiently large $\tau_0$ \hmr{where} the largeness of $\tau_0$
can also be estimated explicitly as a function of $L$. Assertion on the observability is equivalent to the fact that for $L \not \in \cN$ and for $T \ge \tau_0$,  there exists a positive constant $C_L$ computable as a function of $L$ such that for all $y_0 \in L^2(0, L)$, there exists $u \in L^2 (0, T)$ such that $y(T, \cdot) = 0$ in $(0, L)$ and 
$$
\| u \|_{L^2(0, T)} \le C_L \| y_0\|_{L^2(0, T)}. 
$$
This result follows directly from Assertion a) of \Cref{thm-LN}. In \cite{KX21}, the authors also discuss the behaviour of this constant when $L$ is close to a critical length.    
\end{remark}

\section{The cost of fast controls of the nonlinear KdV system - Proof of \Cref{thm-NL}} \label{sect-thm-NL}

\noindent {\it Step 1}: Proof of Assertion $a)$. It suffices to prove that for any $y_T \in L^2(0, L)$ with $\| y_T \|_{L^2(0, L)} \le \eps$ there exists a control $u \in L^2(0, T)$ which brings $y_0 = 0$ at time $0$ to $y_T$ at time $T$ such that 
\be \label{Step1-claim}
\| u \|_{L^2(0, T)} \le C e^{\frac{c}{T^{1/2}}} \| y_0 \|_{L^2(0, L)}. 
\ee
\hmr{since, with $\ty(t, x) = y(T-t, L-x)$ in $(0, T) \times (0, L)$, it holds  
\be
\left\{ \begin{array}{cl}
\ty_t + \ty_x + \ty_{xxx} + \ty \ty_x= 0 \mbox{ in } (0, T) \times (0, L), \\[6pt]
\ty(\cdot, 0) = \ty(\cdot, L) = 0 \mbox{ in } (0, T), \\[6pt]
\ty(0,  \cdot) = y_0(L- \cdot) \mbox{ in } (0, L), \\[6pt]
\ty(T,  \cdot) = 0 \mbox{ in } (0, L),
\end{array} \right. 
\ee
and, by \cite[Proposition 3.4]{Ng-KdV25} (see also \Cref{lem-KdV2}),  
$$
\|\ty_x(\cdot, L)\|_{L^2(0, T)} = \|y_x(\cdot, 0)\|_{L^2(0, T)} \le C  \| u \|_{L^2(0, T)} \mathop{\le}^{\eqref{Step1-claim}} C e^{\frac{c}{T^{1/2}}} \| y_0 \|_{L^2(0, L)}. 
$$
}

In what follows of Step 1, we will prove this fact.  We first involve the Hilbert uniqueness principle. Since the system is controllable in any positive time, we have 
\be
\int_{0}^T |\Phi_x(t, L)|^2 \ge C \int_0^L |\Phi(T, x)|^2 \, d x, 
\ee
for some positive constant $C = C(T, L)$. Here $\varphi \in X_T$ is the unique solution of the backward linearized KdV system 
\be \label{sys-KdV-varphi}
\left\{ \begin{array}{cl}
\Phi_t + \Phi_x + \Phi_{xxx} = 0 \mbox{ in } (0, T) \times (0, L), \\[6pt]
\Phi(\cdot, 0) = \Phi(\cdot, L) = \Phi_x(\cdot, 0)= 0 \mbox{ in } (0, T), \\[6pt]
\Phi(T,  \cdot) \in L^2(0, L). 
\end{array} \right. 
\ee
Moreover, given $y_T \in L^2(0, T)$, one can choose a control of the form 
$$
u(t) = \xi_x(t, L) \mbox{ in } (0, T), 
$$
where $\xi \in X_T$ is a solution of \eqref{sys-KdV-varphi} which is chosen such that 
\be
\int_0^T \varphi_x (t, L) \xi_x(t, L) \, dt =  \int_0^L y_T (x) \varphi(T, x) \, dx
\ee
for all solution $\varphi \in X_T$ of \eqref{sys-KdV-varphi}. Moreover, there is a unique solution $\xi \in X_T$ of \eqref{sys-KdV-varphi}
satisfying this property since $L \not \in \cN$. We also have 
$$
\int_0^T |\xi_x(t, L) |^2 \, dt \le \int_0^T |u(t)|^2 \, dt 
$$
for every control $u \in L^2(0, T)$ steering $0$ at time $0$ to $y_T$ at time $T$. 

Define 
$$
\cL : L^2(0, L) \to L^2(0, T)
$$
by
$$
\cL y_T = \xi_x(t, L). 
$$
Then $\cL$ is linear and continuous. Moreover, by \Cref{thm-LN}, we have
$$
\| \cL \| \le C_2 e^{\frac{c_2}{T^{\frac{1}{2}}}}.  
$$

Given $v \in L^2(0, T)$ with small norm, we denote $\cP v$ the solution at time $T$ of the nonlinear KdV system \eqref{sys-KdV-NL} with the control $v$ and with zero initial data.  In what follows, we assume that $y_T \not \equiv 0$ in $(0, L)$ since there is nothing to prove otherwise. Consider the map, with $\eps = \| y_T \|_{L^2(0, T)}/4$,  
$$
\cG : \Big\{\ty_T \in L^2(0, L); \|\ty_T - y_T \|_{L^2(0, L)} \le \eps \Big\} \to L^2(0, L),  
$$
defined by 
$$
\cG(\ty_T) = \ty_T - \cP \cL \ty_T + y_T. 
$$

One has 
$$
\|\cG (\ty_T) - y_T\|_{L^2(0, L)} = \| \ty_T - \cP \cL \ty_T \|_{L^2(0, L)} \le C \| \ty_T \|_{L^2(0, T)}^2. 
$$
Therefore, for small $\eps$,  
$$
\cG : \Big\{\ty_T \in L^2(0, L); \|\ty_T - y_T \|_{L^2(0, L)} \le \eps \Big\} \to  \Big\{\ty_T \in L^2(0, L); \|\ty_T - y_T \|_{L^2(0, L)} \le \eps \Big\}. 
$$\

We next prove that $\cG$ is a contracting map. Indeed, for $\ty_T, \hy_T \in L^2(0, L)$ with $\|\ty_T - y_T \|_{L^2(0, L)} \le \eps$ and $\|\hy_T - y_T \|_{L^2(0, L)} \le \eps$, we have 
$$
\cG (\ty_T) - \cG (\hy_T) =  \ty_T - \cP \cL \ty_T - \Big( \hy_T - \cP \cL \hy_T \Big) = \ty_T - \hy_T - \Big( \cP \cL \ty_T- \cP \cL \hy_T \Big). 
$$
Let $\tz, \hz \in X_T $ be the solutions of the system 
\be 
\left\{ \begin{array}{cl}
\tz_{t} + \tz_{x} + \tz_{xxx}  + \tz \tz_x =  0  \mbox{ in } (0, T) \times (0, L), \\[6pt]
\tz(\cdot, 0) = \tz(\cdot, L) = 0 \mbox{ in } (0, T), \\[6pt]
\tz_{x}(\cdot, L)  = \cL \ty_T \mbox{ in } (0, T), \\[6pt]
\tz(0,  \cdot) = 0 \mbox{ in } (0, L), 
\end{array} \right. 
\ee
\be
\left\{ \begin{array}{cl}
\hz_{t} + \hz_{x} + \hz_{xxx} + \hz \hz_x =  0  \mbox{ in } (0, T) \times (0, L), \\[6pt]
\hz(\cdot, 0) = \hz(\cdot, L) = 0 \mbox{ in } (0, T), \\[6pt]
\hz_{x}(\cdot, L)  = \cL \hy_T \mbox{ in } (0, T), \\[6pt]
\hz(0,  \cdot) = 0 \mbox{ in } (0, L), 
\end{array} \right. 
\ee
and denote $z \in X_T$ the solution of the system 
\be \label{thm-NL-Step1-z}
\left\{ \begin{array}{cl}
z_{t} + z_{x} + z_{xxx} =  \tz \tz_x - \hz \hz_x   \mbox{ in } (0, T) \times (0, L), \\[6pt]
z(\cdot, 0) = z(\cdot, L) = 0 \mbox{ in } (0, T), \\[6pt]
z_{x}(\cdot, L)  = 0  \mbox{ in } (0, T), \\[6pt]
z(0,  \cdot) = 0 \mbox{ in } (0, L). 
\end{array} \right. 
\ee
We derive that 
\be
\tz(T, \cdot) = \cP \cL \ty_T, \quad \hz(T, \cdot) = \cP \cL \hy_T, \quad \mbox{ and } \quad z(T, \cdot) = \ty_T - \hy_T - \Big( \cP \cL \ty_T- \cP \cL \hy_T \Big). 
\ee
Applying \Cref{lem-KdV2}, we derive from \eqref{thm-NL-Step1-z} that 
\be
\| z\|_{X_T} \le C \| \tz \tz_x - \hz \hz_x \|_{L^1((0, T); L^2(0, L))}. 
\ee
\hmr{Since 
\begin{multline*}
\| \tz \tz_x - \hz \hz_x \|_{L^1((0, T); L^2(0, L))} = \| (\tz - \hz) \tz_x +  \hz (\tz_x -  \hz_x) \|_{L^1((0, T); L^2(0, L))} \\[6pt]
\le  \| (\tz - \hz) \tz_x \|_{L^1((0, T); L^2(0, L))} +  \|\hz (\tz_x -  \hz_x) \|_{L^1((0, T); L^2(0, L))} \\[6pt]
\le C \Big( \| \tz - \hz\|_{X_T} \|\tz \|_{X_T} + \| \hz \|_{X_T} \| \tz_x -  \hz\|_{X_T} \Big), 
\end{multline*}
it follows that} 
\be \label{thm-NL-Step1-z-est}
\| z\|_{X_T} \le C \big(\|\tz  - \hz \|_{X_T} \big)\big(\|\tz\|_{X_T}  + \| \hz \|_{X_T} \big). 
\ee
Since 
$$
\| \tz \|_{X_T} \le C \eps \quad \mbox{ and } \quad \| \hz \|_{X_T} \le C \eps, 
$$
we derive form \eqref{thm-NL-Step1-z-est} that $\cG$ is contracting for $\eps$ sufficiently small. There thus exists $\ty_T \in \Big\{\ty_T \in L^2(0, L); \|\ty_T - y_T \|_{L^2(0, L)} \le \eps \Big\}$ such that 
\be
\cG (\ty_T) = \ty_T. 
\ee
This is equivalent to say that 
$$
\cP \cL \ty_T = y_T. 
$$
Assertion a) then  follows. 

\medskip 

\noindent{\it Step 2:} Proof of Assertion $b)$ with $\varphi = \varphi_1$. Let $y_1, y_2 \in X_T$ be the unique solution of the following systems 
\be 
\left\{ \begin{array}{cl}
y_{1, t} + y_{1, x} + y_{1, xxx} =  0  \mbox{ in } (0, T) \times (0, L), \\[6pt]
y_1(\cdot, 0) = y_1(\cdot, L) = 0 \mbox{ in } (0, T), \\[6pt]
y_{1, x}(\cdot, L)  = u (t) \mbox{ in } (0, T), \\[6pt]
y_1(0,  \cdot) = y_0 \mbox{ in } (0, L), 
\end{array} \right. 
\ee
\be
\left\{ \begin{array}{cl}
y_{2, t} + y_{2, x} + y_{2, xxx} = - y y_x  \mbox{ in } (0, T) \times (0, L), \\[6pt]
y_2(\cdot, 0) = y_2(\cdot, L) = 0 \mbox{ in } (0, T), \\[6pt]
y_{2, x}(\cdot, L)  = 0 \mbox{ in } (0, T), \\[6pt]
y_2(0,  \cdot) = 0 \mbox{ in } (0, L). 
\end{array} \right. 
\ee
Then 
\be \label{thm-NL-Step2-y1} 
y = y_1 + y_2. 
\ee
Let $u_2 \in L^2(0, T)$ be a control which brings $0$ at time $0$ to $y_2(T, \cdot)$ at time $T$ such that 
\be \label{thm-NL-Step2-p1} 
\| u_2 \|_{L^2(0, T)} \le C_1 e^{\frac{c_1}{T^{1/2}}} \| y_2 (T, \cdot) \|_{L^2(0, L)}. 
\ee
Such an $u_2$ exists by \Cref{thm-LN}. 

Let $\ty \in X_T$ be the unique solution of the system 
\be 
\left\{ \begin{array}{cl}
\ty_{t} + \ty_{x} + \ty_{xxx} =  0  \mbox{ in } (0, T) \times (0, L), \\[6pt]
\ty(\cdot, 0) = \ty(\cdot, L) = 0 \mbox{ in } (0, T), \\[6pt]
\ty_{x}(\cdot, L)  = u (t) + u_2 (t) \mbox{ in } (0, T), \\[6pt]
\ty (0,  \cdot) = y_0 \mbox{ in } (0, L). 
\end{array} \right. 
\ee
From the required control property of $u_2$, one derives that 
$$
\ty(T, x) = y_1(T, x)  + y_2 (T, x) \mathop{=}^{\eqref{thm-NL-Step2-y1}} y(T, x) = 0 \mbox{ in } (0, L).   
$$
It follows from \Cref{thm-LN} that 
$$
\| u + u_2 \|_{L^2(0, T)} \ge C_2 e^{\frac{c_2}{T^{1/2}}} \| y_0 \|_{L^2(0, L)}. 
$$
This yields 
\begin{multline} \label{thm-NL-Step2-p2} 
\| u\|_{L^2(0, T)} \ge C_2 e^{\frac{c_2}{T^{1/2}}} \| y_0 \|_{L^2(0, L)} - \| u_2 \|_{L^2(0, T)} \\[6pt]
\mathop{\ge}^{\eqref{thm-NL-Step2-p1}} C_2 e^{\frac{c_2}{T^{1/2}}} \| y_0 \|_{L^2(0, L)}  - C_1 e^{\frac{c_1}{T^{1/2}}} \| y_2 (T, \cdot) \|_{L^2(0, L)}. 
\end{multline}
On the other hand, we have, by \Cref{lem-KdV2},  
\be \label{thm-NL-Step2-p3} 
\| y_2 (T, \cdot) \|_{L^2(0, L)} \le C \| y y_x \|_{L1((0, T); L^2(0, L))}  \le  C \| y \|_{X_T}^2 \le C \big(\| u\|_{L^2(0, T)} + \| y_0 \| \big)^2, 
\ee
The conclusion now follows from \eqref{thm-NL-Step2-p2} and \eqref{thm-NL-Step2-p3} for $\eps_1$ sufficiently small. 

The proof is complete. \qed

\appendix 

\section{Proof of \Cref{pro1}}

This section, consisting of two subsections, is organized as follows. In the first subsection, we present some results used in the proof of \Cref{pro1}. The proof of \Cref{pro1} is given in the second subsection. 

\subsection{Some useful lemmas}

The first result of this section is a variant of \cite[Lemma 4.1]{Ng-cost-fractional-SH} with a similar proof.  

\begin{lemma}  \label{lem-Phi}  Let $\alpha > 1$, $a, b  > 0$,  and let $(\lambda_{1, k})_{k \ge 1}$ and $(\lambda_{2, k})_{k \ge 1}$ be two positive sequences. 
Assume that $\gamma > 0$ and $\Gamma_1, \Gamma_2 < + \infty$ where 
\be \label{lem-Phi-assumption}
\begin{array}{c}
\gamma = \inf_{k \neq n} \min\Big\{ |\lambda_{1, k} - \lambda_{1, n}|, |\lambda_{2, k} - \lambda_{2, n}| \Big\} , \quad  \Gamma_1 = \sup_{k} \max\left\{ \frac{|\lambda_{1, k} - a k^\alpha|}{k^{\alpha - 1}}, \frac{|\lambda_{2, k} - b k^\alpha|}{k^{\alpha - 1}}\right\}, \\[6pt]\Gamma_2 = \sup_{k} \max\left\{ \frac{k^\alpha}{\lambda_{1, k}}, \frac{k^\alpha}{\lambda_{2, k}} \right\}.  
\end{array}
\ee
Set, for $k \ge 1$,  
$$
\lambda_k = \lambda_{1, k}  \quad \mbox{ and } \quad \lambda_{-k}  = - \lambda_{2, k}, 
$$
and, for $n \in \mZ \setminus \{0 \}$,  
$$
\Phi_n (z) =\prod_{k \in \mZ \setminus \{0, n\}} \left( 1 - \frac{z}{\lambda_k - \lambda_n} \right) \mbox{ for } z \in \mC. 
$$
We have 
\be \label{lem-Phi-cl2}
|\Phi_n (z)| \le C e^{c  |z|^{\frac{1}{\alpha}}}  \mbox{ for large $|z|$ with $z \in \mC$,}   
\ee
for some positive constants $c$ and  $C$ depending only on $a$, $\gamma$, $\Gamma_1$, and $\Gamma_2$. \end{lemma}

\begin{proof}  We only prove \eqref{lem-Phi-cl2} for $|z|> 1$. The other case is simpler and omitted. In what follows in this proof, $C$ denotes a positive constant depending only on $a$, $\gamma$, $\Gamma_1$, and  $\Gamma_2$, and can change from one place to another. 
Given $z \in \mC$ with $|z| > 1$, let $m \in \N$ be such that $m^\alpha \le |z| < (m+1)^\alpha$. 

In what follows, we  consider $n, k \in \mZ \setminus \{0\}$.  We first consider the case $|\lambda_n| <  |z|$. We have
\begin{multline}\label{lem-Phi-p2-1}
\ln |\Phi_n (z)| \le C \sum_{|\lambda_k-\lambda_n|  \ge C |z|} \frac{|z|}{|\lambda_k-\lambda_n|} + C \sum_{|\lambda_k-\lambda_n|  < C |z|} \ln \left(1 +  \frac{|z|}{|\lambda_k - \lambda_n|} \right) \\[6pt] 
\le \frac{C |z|}{|z|^{\frac{\alpha - 1}{\alpha}}} + C \sum_{1 \le n^{\alpha-1} m \le C |z|}  \ln \left( 1+  \frac{C |z|}{n^{\alpha-1}m} \right) + C.   
\end{multline}
We have, with $a = C |z|/ n^{\alpha-1} \ge 1$,  
\begin{multline}\label{lem-Phi-p2-2}
\int_1^{a} \ln (1 + a/ t) \, dt = \int_1^{a} \Big( \ln (a + t) - \ln t \Big) \, dt = \int_{1 + a}^{2a} \ln t \, dt - \int_1^{a} \ln t \, dt \\[6pt]
=  \int_{a+1}^{2a} \ln t \, dt -\int_1^{a} \ln t \, dt  = 2 a \ln (2a) - (a+1) \ln (1 + a) - a \ln a - \int_{a+1}^{2a} 1 \, dt + \int_1^{a} 1 \, dt 
\le C a. 
\end{multline}
Combining \eqref{lem-Phi-p2-1} and \eqref{lem-Phi-p2-2} yields 
$$
\ln |\Phi_n (z)| \le C  |z|^{1/ \alpha} + C,  
$$
which implies \eqref{lem-Phi-cl2}.

We next deal with the case $|z| \le |\lambda_n|$. We have
\begin{multline}\label{lem-Phi-p2-3}
\ln |\Phi_n (z)| \le C \sum_{|\lambda_k|  \ge 2 |\lambda_n| } \frac{|z|}{|\lambda_k|} + C \sum_{|\lambda_k|  \le 2 |\lambda_n|} \ln \left(1 +  \frac{|z|}{|\lambda_k - \lambda_n|} \right) + C \\[6pt] 
\le \frac{C |z|}{n^{\alpha-1}} + C \sum_{1 \le m \le C n}  \ln \left( 1+  \frac{C |z|}{n^{\alpha-1} m} \right) + C.   
\end{multline}
We have, with $a = C |z|/ n^{\alpha-1} $ and $\hat n = C n$,  
\begin{multline}\label{lem-Phi-p2-4}
\int_1^{\hat n} \ln (1 + a/ t) \, dt = \int_1^{\hat n} \Big( \ln (a + t) - \ln t \Big) \, dt = \int_{1 + a}^{\hat n+ a} \ln t \, dt - \int_1^{\hat n} \ln t \, dt \\[6pt]
=  \int_{\hat n}^{\hat n+ a} \ln t \, dt -\int_1^{1+ a} \ln t \, dt  = (\hat n+ a) \ln (\hat n+a) - \hat n \ln \hat n  - (a+1) \ln (1 + a)  - \int_{\hat n}^{\hat n+a} 1 \, dt + \int_1^{1+ a} 1 \, dt 
\le C a. 
\end{multline}
Combining \eqref{lem-Phi-p2-3} and \eqref{lem-Phi-p2-4} yields, since $|z| \le |\lambda_n|$,  
$$
\ln |\Phi_n (z)| \le C  |z|^{1/ \alpha} + C,  
$$
which implies \eqref{lem-Phi-cl2}. 
\end{proof}

We next recall the following result in \cite{Ng-cost-fractional-SH} (see also \cite{TT07} in the case $\nu = 1$).

\begin{lemma}\label{lem-H} Let $\mu > 0$, $\nu \ge 1$, and $0 < \beta \le 1 $. Set  
$$
H(z) = \int_{-1}^1 \sigma(t) e^{- i \beta z t } \, dt \mbox{ for } z \in \mC, 
$$
where 
$$
\sigma (t) = e^{- \frac{\nu^\mu}{(1 - t)^\mu} -  \frac{\nu^\mu}{(1 + t)^\mu}} \mbox{ for } t \in (-1, 1).  
$$
We have 
\be \label{lem-H-cl1}
|H(x)| \le C e^{- c |\beta \nu x|^{\frac{\mu}{\mu + 1}}} \mbox{ for } x \in \mR, 
\ee
for some positive constants $C, c$ depending only on $\mu$. 
\end{lemma}

\subsection{Proof of \Cref{pro1}} Multiplying the equation of $y$ in \eqref{sys-KdV-P} by $e^{i \lambda_k t}\varphi_k (x)$ and integrating by parts, one obtains 
\be
\int_0^L \Big(  y(T, x) e^{i \lambda_k T}\varphi_k (x) - y(0, x) \varphi_k (x) \Big) \, dx = \int_0^T v(t) e^{i \lambda_k t} \varphi_{k, x} (L) \, dt 
\ee
Thus $v \in L^2(0, T)$ is a control which steers $y_0$ at time $0$ to $0$ at time $0$ if and only if  
\be \label{pro1-Step1-p1}
 \int_0^T v(t) e^{i \lambda_k t} \, dt = - \frac{1}{\varphi_{k, x}(L)} \int_0^L \varphi_k (x) y_0 (x) \, dx \mbox{ for } k \in \mZ \setminus \{0 \}. 
\ee
Define $w(t)= v(t + T/2)$ for $t \in (-T/2, T/2)$. Then \eqref{pro1-Step1-p1} is equivalent to the fact 
\be \label{pro1-Step1-p2}
 \int_{-T/2}^{T/2} w(t) e^{i \lambda_k t} \, dt = - \frac{e^{- i \lambda_k T/2}}{\varphi_{k, x}(L)} \int_0^L \varphi_k (x) y_0 (x) \, dx. 
\ee

\noindent {\it Step 1:} Proof of \hmr{\eqref{pro1-cl1}}. 

We search a function $W$ such that 
$$
W(- \lambda_k ) = c_k \mbox{ for } k \in \mZ \setminus \{0\},
$$
where 
\be
W(z) = \int_{-T/2}^{T/2} w(t) e^{- i z t} \, dt
\ee
and 
\be \label{pro1-Step1-ck}
c_k =   - \frac{e^{- i \lambda_k T/2}}{\varphi_{k, x}(L)} \hmr{a_k}  \hmr{\mbox{ with } a_k = \int_0^L y_0 \varphi_k \, dx}. 
\ee
We have, by \eqref{cond-N-2}, \hmr{ for $k \in \mZ \setminus \{0\}$,}
\be \label{pro1-Step1-p0}
|c_k| \le \frac{C \hmr{|a_k|}}{|k|}.
\ee

Denote 
\be
\mu = 1/2 \quad \mbox{ and } \quad \alpha = 3. 
\ee
Set, for $n \in \mZ \setminus \{0\}$, 
$$
\Psi_n (z) = \Phi_n (z - \lambda_n) = \prod_{k \neq n} \left(\frac{1 - z/\lambda_k}{1 - \lambda_n/ \lambda_k} \right) \mbox{ for } z \in \mC, 
$$
and 
$$
g_n (z) = \Psi_n(-z) H (z + \lambda_n) \mbox{ for } z \in \mC,  
$$
where, with $\beta = T/2$ and $\nu \beta = \gamma > 0$ sufficiently large but fixed, 
$$
H(z) = \alpha_0 \int_{-1}^1 \sigma(t) e^{- i \beta z t } \, dt \mbox{ with } 
\sigma (t) = e^{- \frac{\nu^\mu}{(1 - t)^\mu} -  \frac{\nu^\mu}{(1 + t)^\mu}} \mbox{ for } t \in (-1, 1),
$$
where $\alpha_0$ is chosen such that $H(0) = 1$.  

Then 
$$
\alpha_0 \le C e^{\hmr{3} \nu^\mu} = C e^{\hmr{3} \nu^{\frac{1}{\alpha - 1}}}
$$
and 
$$
g_n(-\lambda_k)  = \delta_{n, k}. 
$$

We have, by  \Cref{lem-Phi,lem-H},  
$$
|g_n(x)| \le C e^{\hmr{3} \nu^{\frac{1}{\alpha - 1}}} e^{ c_1 |x+ \lambda_n|^{\frac{1}{\alpha}}} e^{-c_2 ( \gamma |x + \lambda_n|) ^\frac{1}{\alpha}}  \mbox{ for } x \in \mR, 
$$
for some positive constant $c_1$, $c_2$, and $C$ independent of $n$. 

This implies, since $\gamma$ is sufficiently large (the largeness of $\gamma$ is independent of $n$), 
\be \label{pro1-Step1-pro-gn}
|g_n(x)| \le C e^{\hmr{3} \nu^{\frac{1}{\alpha - 1}}} e^{ - |x+ \lambda_n|^{\frac{1}{\alpha}}}  \mbox{ for } x \in \mR. 
\ee

We now define 
$$
W(z) = \sum_{n \in \mZ \setminus\{0 \}} c_n g_n (z) \mbox{ for } z \in \mC. 
$$
Then  
\be
\| W \|_{L^2(\mR)} \le \sum_{n \hmr{\in} \mZ \setminus\{0 \}} |c_n| \|g_n \|_{L^2(\mR)} \mathop{\le}^{\eqref{pro1-Step1-p0}}  C e^{\hmr{3} \nu^{\frac{1}{\alpha - 1}}} \sum_{n \hmr{\in} \mZ \setminus\{0 \} } \frac{|a_n|}{|n|} \le C e^{\hmr{3} \nu^{\frac{1}{\alpha - 1}}}  \| y_0\|_{L^2(0, 1)}. 
\ee
We also have, by  \Cref{lem-Phi}, 
$$
|W(z)| \le \sum_{n \hmr{\in} \mZ \setminus\{0 \}} |c_n| |g_n(z)| \le C_{T, \alpha} \sum_{n \hmr{\in} \mZ \setminus\{0 \} } |c_n| e^{|z| T/2} e^{c |z|^{\frac{1}{\alpha}}} \le C_{T, \alpha, \eps} e^{|z| (T/2 + \eps)}  \| y_0\|_{L^2(0, 1)} \mbox{ for } z \in \mC,  
$$
for all $\eps > 0$.  It follows from Paley-Wiener's theorem, see, e.g., \cite[19.3 Theorem]{Rudin-RC},  that there exists $w \in L^2(\mR)$ with $\supp w \subset [-T/2, T/2]$ such that 
\be
W(z) = \int_{-T/2}^{T/2} w(t) e^{- i z t} \, dt,
\ee
\be
\| w\|_{L^2(\mR)} \le C e^{ \frac{C}{T^{\frac{1}{\alpha - 1}}}} \| y_0\|_{L^2(0, 1)}, 
\ee
and 
\be
W(-\lambda_k) = c_k. 
\ee
The conclusion follows by taking $v(t) = w(t - T/2)$ for $t \in (0, T)$. \qed

\medskip 
{\color{black}
\noindent {\it Step 2:} Proof of \eqref{pro1-cl2} of \Cref{pro1}. 
Let $\ty_N$ be the solution of \eqref{sys-KdV-P} with zero initial data and with $v = v_N$ given in Step 1 with the sum of $k$ for $|k| < N$, i.e., 
\be \label{pro1-Step2-vN}
v_N(t) = w_N(t - T/2) \mbox{ with } \hat w_N (z) = \sum_{k \in \mZ \setminus \{0 \}; |k| \le N} c_n g_n(z) \mbox{ in } \mR. 
\ee
Then $v_N \in H^1(\mR)$ with $\supp v_N \subset [0, T]$,  and, for $0 < \tau < T$,  
\be
\int_0^L  \ty_N(\tau, x) e^{i \lambda_k \tau}\varphi_k (x)  \, dx =  \int_0^\tau v_N(t) e^{i \lambda_k t} \varphi_{k}' (L) \, dt. 
\ee
It follows that
\begin{multline}
 \int_0^L  \ty_N(\tau, x) \varphi_k (x)  \, dx
=    \varphi_{k}' (L)   \int_0^\tau v_N(t) e^{i \lambda_k (t - \tau)} \, dt \\[6pt]
=  - \frac{\varphi_{k}' (L)}{i \lambda_k}   \int_0^\tau v_N'(t) e^{i \lambda_k (t - \tau)} \, dt  =  - \frac{\varphi_{k}' (L)}{i \lambda_k}  v_N'* \rho_k (\tau), 
\end{multline}
where 
\be \label{pro1-rho}
\rho_k (t) =  e^{i \lambda t} \mathds{1}_{0 \le t \le T} \mbox{ for } t \in \mR. 
\ee

Since 
$$
\ty_N(t, \xi) = \sum_{k \in \mZ \setminus \{ 0\}} \left( \int_0^L  \ty_N(\tau, x) \varphi_k (x)  \, dx \right) \overline{\varphi_k} (\xi), 
$$
we derive that 
\be \label{pro1-Step2-pp0}
\int_0^T |\partial_x \ty_N(\tau, L) |^2 \, d \tau \le C  \int_\mR \left| \sum_{k \in \mZ \setminus \{0 \}} \frac{|\varphi_{k}'|^2 (L)}{\lambda_k}  v_N'* \rho_k\right|^2 \, d \tau. 
\ee
By Parseval's theorem, we obtain
\be
 \int_\mR \left|  \sum_{k \in \mZ \setminus \{0 \}} \frac{|\varphi_{k}'|^2 (L)}{\lambda_k}  v_N'* \rho_k\right|^2 \, d \tau = C \int_{\mR} |\xi \hat v_N f |^2 \, d \xi, 
\ee 
where
\be \label{pro1-Step2-f}
f(\xi) = \sum_{k \in \mZ \setminus \{0 \}} \frac{|\varphi_{k}'|^2 (L)}{\lambda_k}  \hat \rho_k (\xi) \mbox{ for } \xi \in \mR.   
\ee

Note that 
$$
\left|\frac{|\varphi_{k}'|^2 (L)}{ \lambda_k} \right| \mathop{\le}^{\eqref{cond-N-3}} \frac{C}{|k|} \quad \mbox{ and } \quad \hat \rho_k (\xi) \mathop{=}^{\eqref{pro1-rho}} \frac{e^{i (\lambda_k - \xi) T} - 1}{i (\lambda_k - \xi)} \mbox{ for } \xi \in \mR.  
$$
For $\xi \in (\lambda_n, \lambda_{n+1})$ and for sufficiently large positive $n$, we have  
\begin{multline} \label{pro1-Step2-f1}
|f(\xi)| = \sum_{k \in \mZ \setminus \{0 \}} \left| \frac{|\varphi_{k}'|^2 (L)}{\lambda_k} \right| | \hat \rho_k (\xi) | \\[6pt]\le  \frac{C}{n} \left| \frac{e^{i (\lambda_n - \xi) T} - 1}{i (\lambda_n - \xi)} \right|  +  \frac{C}{n} \left| \frac{e^{i (\lambda_{n+1} - \xi) T} - 1}{i (\lambda_{n+1} - \xi)} \right| + \sum_{k \ge 1} \frac{C}{(n+k)^2 k n} + \frac{C}{n^2} \le \frac{C}{n},  
\end{multline}
since 
$$
\sum_{k \in \mZ \setminus \{0 \}}  = \sum_{k = n, n+1} +  \sum_{k >  n+1} + \sum_{k < n-1, k \neq 0}. 
$$

Similarly, for $\xi \in (\lambda_n, \lambda_{n+1})$ and for sufficiently large negative $n$, 
\be \label{pro1-Step2-f2}
|f(\xi)| = \sum_{k \in \mZ \setminus \{0 \}} \left| \frac{|\varphi_{k}'|^2 (L)}{\lambda_k} \right| | \hat \rho_k (\xi) | \le \frac{C}{|n|}. 
\ee

Since 
\begin{multline*}
a_k = \int_0^L y_0 \varphi_k \, dx = \frac{1}{i \lambda_k} \int_0^L y_0 (\varphi_k''' + \varphi_k') \, dx  = -  \frac{1}{i \lambda_k} \int_0^L y_0' (\varphi_k'' + \varphi_k) \, dx \\[6pt] = - \frac{1}{ i \lambda_k} \left(-\int_0^L y_0'' \varphi_k' \, dx + \varphi_k' (L) (y_0' (L) - y_0' (0)) + \int_0^L y_0' \varphi_k \, dx  \right), 
\end{multline*}
we derive from \eqref{pro1-Step1-ck} that 
\be
|c_k| \le \frac{C}{|k|^3} \| y_0 \|_{H^2(0, L)}. 
\ee
Since, by \eqref{pro1-Step1-pro-gn} and \eqref{pro1-Step2-vN},  
$$
|\xi \hat v_N(\xi)| \le \sum_{k \in \mZ \setminus \{0 \}; |k| \le N}  C e^{3 \nu^{\frac{1}{\alpha - 1}}}  |\xi| |c_k| e^{- c |\xi + \lambda_k|^{\frac{1}{\alpha}}} \le \sum_{k \in \mZ \setminus \{0 \}; |k| \le N} C e^{3 \nu^{\frac{1}{\alpha - 1}}}   \frac{|\xi|}{|k|^3} e^{- c |\xi + \lambda_k|^{\frac{1}{\alpha}}},  
$$
it follows from \eqref{pro1-Step2-f}, \eqref{pro1-Step2-f1} and  \eqref{pro1-Step2-f2}  that 
\begin{multline} \label{pro1-Step2-pp1}
\int_{\mR} |\xi \hat v_N f |^2 \, d \xi \\[6pt]
\le \left( \sum_{n \in \mZ; |n| \ge 2} \frac{C}{n^2}  \int_{\xi \in (\lambda_n, \lambda_{n+1})} e^{6 \nu^{\frac{1}{\alpha - 1}}} \left( \sum_{k \in \mZ \setminus \{0 \}; |k| \le N} \frac{|\xi|}{|k|^3} e^{-  c |\xi + \lambda_k|^{\frac{1}{\alpha}}} \right)^2 \, d \xi+ C e^{6 \nu^{\frac{1}{\alpha - 1}}} \right) \|y_0 \|_{H^2(0, L)} \\[6pt]\le 
C e^{6 \nu^{\frac{1}{\alpha - 1}}} \|y_0 \|_{H^2(0, L)}. 
\end{multline}
Using the same arguments, one also derives that 
$$
\int_{\mR} |\xi (\hat v_N - \hat v_M) f |^2 \, d \xi \to 0 \mbox{ as } M, N \to + \infty.  
$$
It follows that  $\partial_x \ty_N(\cdot, L)$ is a Cauchy sequence in $L^2(0, T)$. The conclusion now follows from  \eqref{pro1-Step2-pp0} and \eqref{pro1-Step2-pp1}.
}

\medskip
\noindent {\it Step 3:} Proof of \hmr{\eqref{pro1-cl3}} of \Cref{pro1}. This is a direct consequence of \cite[Lemma 3.2]{Ng-cost-fractional-SH}. 

The proof is complete. \qed

\begin{remark} \rm \label{rem-pro1} Fix $0 < \ell_1 < \ell_2 < + \infty$ and let $\ell_1 \le L \le \ell_2$. Some comments on \Cref{pro1} are in orders. 

\begin{itemize}
\item The proof of Assertion a) of \Cref{pro1} shows that one can choose the constant $c_1$ independent of $L$ but depends on $\ell_1$ and $\ell_2$ (see the proof of \Cref{lem-Phi}). Nevertheless, the constant $C_1 = C_1 (L)$ blows up as $\mbox{dist} (L, \cN) \to 0$. This blow-up is quantified by the estimate \eqref{pro1-Step1-p0} (see also \eqref{cond-N-1} and \eqref{cond-N-2}) and can be obtained explicitly.   

\item The proof of Assertion b) of \Cref{pro1} shows that one can choose the constant $c_2$ independent of $L$ but depends on $\ell_1$ and $\ell_2$ (see the proof of \cite[Lemma 3.2]{Ng-cost-fractional-SH}). Nevertheless, the constant $C_2 = C_2 (L)$ blows up as $\mbox{dist} (L, \cN) \to 0$. This blow-up is quantified by the estimate 
\eqref{pro1-Step1-p2}, and $\varphi_1$ must be replaced by $\varphi_{k_0}$ for some appropriate $k_0$. 
\end{itemize}
\end{remark}

\end{document}